\documentclass{amsart}
\usepackage{amssymb,amsthm,amscd}
\usepackage{pst-all}

\newcommand{\ov}{\overline}
\newcommand{\rsp}{\raisebox{0em}[2ex][1.3ex]{\rule{0em}{2ex} }}
\newcommand{\Cl}{{\operatorname{Cl}}}

\newcommand{\disc}{{\operatorname{disc}}}

\newcommand{\GL}{{\operatorname{GL}}}
\newcommand{\SL}{{\operatorname{SL}}}
\newcommand{\cA}{\mathcal A}

\newcommand{\cP}{{\mathcal P}}

\newcommand{\Z}{{\mathbb Z}}
\newcommand{\Q}{{\mathbb Q}}
\newcommand{\R}{{\mathbb R}}
\newcommand{\eps}{{\varepsilon}}
\newcommand{\lra}{\longrightarrow}
\newcommand{\Lra}{\Longrightarrow}

\newcommand{\sym}[2]{\big(\frac{#1}{#2}\big)_4^{\phantom{p}}}

\newfont{\cyr}{wncyb10}
\newcommand{\sha}{\mbox{\cyr Sh}}

\newcommand{\matr}[4]{\Big(\,\begin{matrix} 
             #1 & #2 \\ #3 & #4 \end{matrix}\, \Big)}
\newcommand{\smatr}[4]{(\begin{smallmatrix} 
             #1 & #2 \\ #3 & #4 \end{smallmatrix})}

\newcommand{\cube}[8]{%
\begin{picture}(90,70)(0,0)
\put(0,0){\makebox(0,0){$#3$}}
\put(0,45){\makebox(0,0){$#1$}}
\put(45,0){\makebox(0,0){$#4$}}
\put(45,45){\makebox(0,0){$#2$}}
\put(0,8){\line(0,1){30}}
\put(5,0){\line(1,0){30}}
\put(45,8){\line(0,1){28}}
\put(5,45){\line(1,0){30}} 
\put(22,20){\makebox(0,0){$#7$}}
\put(22,65){\makebox(0,0){$#5$}}
\put(65,20){\makebox(0,0){$#8$}}
\put(65,65){\makebox(0,0){$#6$}}
\put(5,5){\line(1,1){10}}
\put(22,29){\line(0,1){10}}
\put(22,49){\line(0,1){10}}
\put(65,29){\line(0,1){28}}
\put(50,5){\line(1,1){10}}
\put(5,50){\line(1,1){10}}
\put(50,49){\line(1,1){10}}
\put(30,65){\line(1,0){25}}
\put(30,20){\line(1,0){10}}
\put(48,20){\line(1,0){7}}
\end{picture}}

\newcommand{\cubeA}[9]{%
\begin{picture}(90,65)(0,0)
\put(0,0){\makebox(0,0){$#3$}}
\put(0,45){\makebox(0,0){$#1$}}
\put(45,0){\makebox(0,0){$#4$}}
\put(45,45){\makebox(0,0){$#2$}}
\put(0,8){\line(0,1){30}}
\put(5,0){\line(1,0){30}}
\put(45,8){\line(0,1){28}}
\put(5,45){\line(1,0){30}} 
\put(22,20){\makebox(0,0){$#7$}}
\put(22,65){\makebox(0,0){$#5$}}
\put(65,20){\makebox(0,0){$#8$}}
\put(65,65){\makebox(0,0){$#6$}}
\put(5,5){\line(1,1){10}}
\put(22,29){\line(0,1){10}}
\put(22,49){\line(0,1){10}}
\put(65,29){\line(0,1){28}}
\put(50,5){\line(1,1){10}}
\put(5,50){\line(1,1){10}}
\put(50,49){\line(1,1){10}}
\put(30,65){\line(1,0){25}}
\put(30,20){\line(1,0){10}}
\put(48,20){\line(1,0){7}}
 \put(-20,30){\makebox(0,0){$#9 = $}}
\end{picture}}

\newtheorem{thm}{Theorem}
\newtheorem{prop}{Proposition}
\newtheorem{lem}{Lemma}
\newtheorem{cor}{Corollary}

\title[Binary Quadratic Forms and the Hasse Principle]
      {Binary Quadratic Forms and Counterexamples to Hasse's
        Local-Global Principle}
\author{F. Lemmermeyer}
\begin{document}

\begin{abstract}
After a brief introduction to the classical theory of binary 
quadratic forms we use these results for proving (most of) the
claims made by P\'epin in a series of articles on unsolvable
quartic diophantine equations, and for constructing families 
of counterexamples to the Hasse Principle for curves of genus $1$ 
defined by equations of the form $ax^4 + by^4 = z^2$.
\end{abstract}

\maketitle
\markboth{Binary Quadratic Forms}{\today}
\begin{center} \today \end{center}

\section*{Introduction}

In a series of four articles (\cite{Pep74, Pep79, Pep80, Pep82}), 
Th\'eophile P\'epin announced the unsolvability of certain 
diophantine equations of the form $ax^4 + by^4 = z^2$. He did
not supply proofs for his claims; a few of his ``theorems'' 
were first proved in the author's article \cite{LPep} using 
techniques that P\'epin was not familiar with, such as the 
arithmetic of ideals\footnote{Ideals were introduced by Dedekind
in 1879, but took off only after Hilbert published his Zahlbericht
in 1897; its French translation \cite{HZB-f} started appearing in 1909.
Kummer had introduced ideal numbers already in the 1840s, but these
were used exclusively for studying higher reciprocity laws and
Fermat's Last Theorem. For investigating diophantine equations, the 
mathematicians of the late 19th century preferred Gauss's theory of 
quadratic forms (see Dirichlet \cite{Dir25,Dir28} and P\'epin \cite{Pep75})
to Dedekind's ideal theory.} in quadratic number fields. 

At the time P\'epin was studying these diophantine equations,
he was working on simplifying Gauss's theory of composition
of quadratic forms (see e.g. \cite{Pep75,PepC}), and it seems natural
to look into the theory of binary quadratic forms for approaches
to P\'epin's results. In fact we will find that all of P\'epin's 
claims (and a lot more) can be proved very naturally using 
quadratic forms.

We start by briefly recalling the relevant facts following 
Bhargava's exposition \cite{Bhar1} of Gauss's theory 
(Cox \cite{Cox} and Flath \cite{Flath} also provide excellent 
introductions) for the following reasons:
\begin{itemize}
\item It gives us an opportunity to point out some classical references on 
      composition of forms that deserve to be better known; this 
      includes work by Cayley, Riss, Speiser and others.
\item Most mathematicians nowadays are unfamiliar with the classical 
      language of binary quadratic forms, and in particular with 
      composition of forms.  
\item We need some results (such as Thm. \ref{T2} below) in a form
      that is slightly stronger than what can be found in the literature.
\item We have to fix the language anyway.
\end{itemize}
In addition, working with ray class groups instead of forms with
nonfundamental discriminants does not save space since we would have
to translate the results into the language of forms for comparing them
with P\'epin's statements.

Afterwards, we will supply the proofs P\'epin must have had in mind.
In the final section we will interpret our results in terms of Hasse's 
Local-Global Principle and the Tate-Shafarevich group of elliptic curves.

\section{Composition of Binary Quadratic Forms}

A binary quadratic form is a form $Q(x,y) = Ax^2 + Bxy + Cy^2$
in two variables with degree $2$ and coefficients $A, B, C \in \Z$;
in the following, we will use the notation $Q = (A,B,C)$.
The discriminant $\Delta = B^2 - 4AC$ of $Q$ will always be 
assumed to be a nonsquare. A form $(A,B,C)$ is called primitive
if $\gcd(A,B,C) = 1$. 

The group $\SL_2(\Z)$ of matrices $S = \smatr{r}{s}{t}{u}$ with 
$r, s, t, u \in \Z$ and determinant $\det S = ru-st = +1$ acts 
on the set of primitive forms with discriminant $\Delta$ via 
$Q|_S = Q(rx+sy,tx+uy)$; two forms $Q$ and $Q'$ are called equivalent
if there is an $S \in \SL_2(\Z)$ such that $Q' = Q|_S$. 
Given $Q = (A,B,C)$, the forms $Q' = Q|_S$ with $S = \smatr{1}{s}{0}{1}$
are said to be parallel to $Q$; their coefficients are
$Q' = (A,B+2As,C')$ with $C'=Q(s,1)$. Observe in particular that
we can always change $B$ modulo $2A$ (and compute the last coefficient
from the discriminant $\Delta$) without leaving the equivalence
class of the form. There are finitely many equivalence classes 
since each form is equivalent to one whose cefficients are bounded 
by $\sqrt{|\Delta|}$. 

A form $Q = (A,B,C)$ represents an integer $m$ primitively if
there exist coprime integers $a, b$ such that $m = Q(a,b)$. 
If $Q$ primitively represents $m$, then there is an $S \in \SL_2(\Z)$
such that $Q|_S = (A',B',C')$ with $A' = m$. In fact, write
$Q(r,t) = m$; since $\gcd(r,t) = 1$, there exist $s, u \in \Z$
with $ru-st = 1$; now set $S = \smatr{r}{s}{t}{u}$. This implies 
that forms representing $1$ are equivalent to the principal form 
$Q_0$ defined by
$$ Q_0(x,y) = \begin{cases}
         (1,0,m) & \text{ if } \Delta = -4m, \\
         (1,1,m) & \text{ if } \Delta = 1-4m
             \end{cases} $$
In fact, forms representing $1$ are equivalent to forms $(1,B,C)$, and
reducing $B$ modulo $2$ shows that they are equivalent to $Q_0$.

The set of $\SL_2(\Z)$-equivalence classes of primitive forms
(positive definite if $\Delta < 0$) can be given a group 
structure by introducing composition of forms, which can be most
easily explained using Bhargava's cubes\footnote{Historically, 
Bhargava's cubes occurred in the form of 
$2 \times 2 \times 2$-hypermatrices in the work of Cayley 
(see \cite{Cay45,CayHD}, or, for a modern account, 
\cite[Chap. 14, Prop. 1.4]{GKZ}), as pairs of bilinear forms
as in Eqn. (\ref{Ebf1}) and (\ref{Ebf2}) below (see 
Gauss \cite{Gauss} and Dedekind \cite{Dede}), as a trilinear
form (Dedekind \cite{Dede} and Weber \cite{WebC}), 
and as $2 \times 4$-matrices 
$\big(\begin{smallmatrix} a & b & c & d \\ 
                          e & f & g & h \end{smallmatrix}\big)$
(see Speiser \cite{Spei}, Riss \cite{Riss}, Shanks \cite{Sha3,Shac1,Shac2}, 
Towber \cite{Tow}, and most other presentations of composition).}. 
Each cube 
$$ \cubeA{a}{b}{c}{d}{e}{f}{g}{h}{\cA} $$
of eight integers $a, b, c, d, e, f, g, h$ can be sliced in three 
different ways (up-down, left-right, front-back):
\begin{align*}
UD & & M_1 = U & = \matr{a}{e}{b}{f}, & N_1 = D & = \matr{c}{g}{d}{h}, \\
LR & & M_2 = L & = \matr{a}{c}{e}{g}, & N_2 = R & = \matr{b}{d}{f}{h}, \\
FB & & M_3 = F & = \matr{a}{b}{c}{d}, & N_3 = B & = \matr{e}{f}{g}{h}.
\end{align*}
To each slicing we can associate a binary quadratic form 
$Q_i = Q_i^\cA$ by putting
$$ Q_i(x,y) = - \det (M_i x + N_i y). $$
Explicitly we find
\begin{align}
\label{EQ1} Q_1(x,y) & = (be-af)x^2 + (bg+de-ah-cf)xy + (dg-ch)y^2, \\
\label{EQ2} Q_2(x,y) & = (ce-ag)x^2 + (cf+de-ah-bg)xy + (df-bh)y^2, \\
\label{EQ3} Q_3(x,y) & = (bc-ad)x^2 + (bg+cf-ah-de)xy + (fg-eh)y^2.
\end{align}
These forms all have the same discriminant, and if two of them are
primitive (or positive definite), then so is the third.

On the set $\Cl^+(\Delta)$ of equivalence classes of primitive
forms with discriminant $\Delta$ we can introduce a group structure
by demanding that 
\begin{itemize}
\item The neutral element $[1]$ is the class of the principal form
      $Q_0(x,y)$. 
\item We have $([Q_1] \cdot [Q_2]) \cdot [Q_3] = [1]$ if and only if there
      exists a cube $\cA$ with $Q_i = Q_i^\cA$ for $i = 1, 2, 3$.
\end{itemize}
Most of the group axioms are easily checked: the cubes
$$  \cube{0}{1}{1}{0}{1}{0}{0}{\ m} \qquad   \text{ or } \qquad
    \cube{0}{1}{1}{1}{1}{1}{1}{\ \mu} $$ 
show that $[Q_0][Q_0][Q_0] = [1]$ in the two cases $\Delta = 4m$ 
and $\Delta = 4m+1 = 4\mu-3$, with $\mu = m+1$.

Next observe that $B \equiv \Delta \bmod 2$; thus we can put $B = 2b$ if
$\Delta = 4m$, and $B = 2b-1$ if $\Delta = 1+4m$. With 
$b' = 1-b$ we then find that the two cubes
$$ \cube{0}{1}{1}{0}{A}{\ -b}{b}{\ -C} \qquad \text{and} \qquad 
    \cube{0}{1}{1}{0}{A}{b'}{b}{\ -C} $$
give rise to the quadratic forms $Q_1 = Q_0$, $Q_2 = (A,B,C)$,
and $Q_3 = (A,-B,C)$. This shows that the inverse of $[Q]$ for 
$Q = (A,B,C)$ is the class of $Q^- = (A,-B,C)$. Note in particular
that, in general, the classes of $Q$ and $Q^-$ are different (in fact
they coincide if and only if their class has order dividing $2$), 
although both $Q$ and $Q^-$ represent exactly the same integers 
since $Q(x,y) = Q^-(x,-y)$. Gauss almost apologized for 
distinguishing the classes of these forms.

The verification of associativity is a little bit involved.
Perhaps the simplest approach uses Dirichlet's method of
united and concordant forms. Two primitive forms 
$Q_1 = (A_1,B_1,C_1)$ and $Q_2 = (A_2,B_2,C_2)$ are called 
concordant if $B_1 = B_2$, $C_1 = A_2C$ and $C_2 = A_1C$ for 
some integer $C$. The composition of $Q_1$ and $Q_2$ then is 
the form $(A_1A_2,B,C)$, as can be seen from the cube
$$ \cubeA{0}{A_1}{1}{0}{A_2}{B}{0}{\ \ -C}{\cA} $$ 
and the associated forms 
\begin{align*}
  Q_1 & = A_1x^2  + Bxy + A_2Cy^2, \\
  Q_2 & = A_2x^2  + Bxy + A_1Cy^2, \\
  Q_3 & = A_1A_2x^2 - Bxy + Cy^2.
\end{align*}
Given three forms, associativity follows immediately if we succeed in 
replacing the forms by equivalent forms with the same middle coefficients,
which is quite easy using the observation that forms represent
infinitely many integers coprime to any given number.

Composing two (classes of) forms requires solving\footnote{For an
excellent account of the composition formulas using Dedekind's 
approach via modules see Lenstra \cite{Len} and Schoof \cite{Schoof}. 
The clearest exposition of the composition algorithm of binary 
quadratic forms not based on modules is probably Speiser's 
\cite{Spei}; his techniques also allow to fill the gaps in Shanks' 
algorithm given in \cite{Shac2}. Shanks later gave a full version 
of his composition algorithm which he called NUCOMP.} systems of 
diophantine equations. All we need in this article 
is the following observation:

\begin{thm}
The $\SL_2(\Z)$-equivalence classes of primitive forms with discriminant 
$\Delta$ (positive definite forms if $\Delta < 0$) form a group with 
respect to composition. If $Q_1 = (A_1,B_1,C_1)$ and $Q_2 = (A_2,B_2,C_2)$ 
are primitive forms with discriminant $\Delta$, and if 
$e = \gcd(A_1,A_2,\frac12(B_1+B_2))$, then we can always find a form 
$Q_3 = (A_3,B_3,C_3)$ in the class $[Q_1][Q_2]$ with $A_3 = A_1A_2/e^2$. 
\end{thm}

The group of $\SL_2(\Z)$ equivalence classes of primitive forms 
with discriminant $\Delta$ is called the class group in the strict
sense and is denoted by $\Cl^+(\Delta)$ (the equivalence classes 
with respect to a suitably defined action by $\GL_2(\Z)$ gives
rise to the class group $\Cl(\Delta)$ in the wide sense; for
negative discriminants, both notions coincide). It is isomorphic 
to the ideal class group in the strict sense of the order with 
discriminant $\Delta$ inside the quadratic number field 
$\Q(\sqrt{\Delta}\,)$.

The connection between Bhargava's group law and Gauss composition 
is provided by the following

\begin{thm}
Let $\cA = [a,b,c,d,e,f,g,h]$ be a cube to which three primitive 
forms $Q_i = Q_i^\cA$ are attached. Then 
\begin{equation}\label{EGcom}
   Q_1(x_1,y_1) Q_2(x_2,y_2) = Q_3(x_3,-y_3), 
\end{equation}
where $x_3$ and $y_3$ are bilinear forms (linear forms in $x_1, y_1$ 
and $x_2, y_2$, respectively) and are given by
\begin{eqnarray}
\label{Ebf1}   x_3 & = ex_1x_2 + fx_1y_2 + gx_2y_1 + hy_1y_2, \\
\label{Ebf2}   y_3 & = ax_1x_2 + bx_1y_2 + cx_2y_1 + dy_1y_2.
\end{eqnarray}
\end{thm}

This can be verified e.g. by a computer algebra system; for a 
conceptual proof, see Dedekind \cite{Dede} or Speiser \cite{Spei}. 
The somewhat unnatural minus sign on the right hand side of 
(\ref{EGcom}) comes from breaking the symmetry between the forms 
$Q_i$; Dedekind \cite{Dede} and Weber \cite{WebC} have shown that 
$$ Q_1(x_1,y_1) Q_2(x_2,y_2) Q_3(x_3,y_3) = Q_0(x_4,y_4) $$
for certain trilinear forms $x_4, y_4$; this formula preserves
the symmetry of the forms involved and makes the group law
appear completely natural.

Gauss defined a form $Q_3$ to be a composite of the forms $Q_1$
and $Q_2$ if the identity (\ref{EGcom}) holds and if (and
this additional condition is crucial -- it is what allowed
Gauss to make form classes into a group\footnote{Composition of 
binary quadratic forms can be generalized to arbitrary rings if 
one is willing to replace quadratic forms by quadratic spaces; 
see Kneser \cite{KneC} and Koecher \cite{Koe87}. Gauss's proof 
that composition gives a group structure extends without problems to 
principal ideal domains with characteristic $\ne 2$, and even
to slightly more general rings (see e.g. Towber \cite{Tow}).}) 
the formulas (\ref{EQ1}) and (\ref{EQ2}) hold.

\begin{cor}\label{C1}
Let $\Delta$ be a discriminant, $r$ an integer, and $p$ a prime not 
dividing $\Delta$. Assume that $p$ is primitively represented by a 
form $Q_1$, and that $pr$ is represented primitively by $Q_2$. Then 
we can choose $g \in \{\pm 1\}$ in such a way that $p^2r$ is represented 
primitively by any form $Q_3$ with $[Q_1][Q_2]^g = [Q_3]$.
\end{cor}

It is obvious from the Gaussian composition formula (\ref{EGcom})
that $ap^2$ is represented by $Q_3$ and $Q_3'$; what we have to 
prove is that there exists a {\em primitive} representation. 

As an example illustrating the problem, take $Q_1 = (2,1,3)$
and $Q_2 = (2,-1,3)$. Both forms represent $p = 3$ primitively:
we have $3 = Q_1(0,1) = Q_2(0,1)$. We also have
$[Q_1][Q_2] = [Q_0]$ and $[Q_1][Q_2]^{-1} = [Q_2]$.
Both $Q_0$ and $Q_2$ represent $9$, but $Q_2(2,1) = 9$ is a 
primitive representation whereas $Q_0(3,0) = 9$ is not.

\begin{proof}[Proof of Cor. \ref{C1}]
We may assume without loss of generality that $Q_1 = (p,B_1,C_1)$ 
and $Q_2 = (pr,B_2,C_2)$. The composition algorithm shows that 
$[Q_1][Q_2] = [Q_3]$ for some form $Q_3 = (A_3,B_3,C_3)$ with 
$A_3 = p^2r/e^2$, where $e = \gcd(p,\frac12(B_1+B_2))$. If 
$p \nmid \frac12(B_1+B_2)$, then $Q_3(1,0) = p^2r$ and we are done.
If $p \mid \frac12(B_1+B_2)$, replace $Q_2$ by $Q_2^- = (pr,-B_2,C_2)$;
in this case, we find $[Q_1][Q_2^-] = [Q_3']$ for 
$Q_3' = (A_3,B_3,C_3)$ with $A_3 = p^2r/e^2$, where 
$e = \gcd(p,\frac12(B_1-B_2))$. If $p \nmid \frac12(B_1-B_2)$ 
we are done; the only remaining problematic case is where
$p$ divides both $\frac12(B_1+B_2)$ and $\frac12(B_1-B_2)$, which
implies that $p \mid B_1$ and $p \mid B_2$. But then 
$p \mid (B_1^2 - 4pC_1) = \Delta$, contradicting our assumption.
\end{proof}

\section{Genus Theory}

Gauss's genus theory characterizes the square classes in $\Cl^+(\Delta)$.
Two classes $[Q_1]$ and $[Q_2]$ are said to be in the same genus if 
there is a class $[Q]$ such that $[Q_1] = [Q_2][Q]^2$. The principal 
genus is the genus containing the principal class $[Q_0]$; by 
definition the principal genus consists of all square classes.

\subsection*{Extracting Square Roots in the Class Group}

Recall that a discriminant $\Delta$ is called fundamental if it is
the discriminant of a quadratic number field. Arbitrary discriminants
can always be written in the form $\Delta = \Delta_0 f^2$, where
$\Delta_0$ is fundamental, and where $f$ is an integer called the
conductor of the ring $\Z \oplus \frac{\Delta + \sqrt{\Delta}}2 \Z$. 
An elementary technique for detecting squares in the class group 
$\Cl(\Delta)$ is provided by the following

\begin{thm}\label{T2}
Let $\Delta_0$ be a fundamental discriminant, and assume that $Q$ is 
a primitive form with discriminant $\Delta = \Delta_0 f^2$. Then
the following conditions are equivalent:
\begin{enumerate}
\item[$i)$] $Q$ represents a square $m^2$ coprime to $f$.
\item[$i')$] $Q$ represents a square $m^2$ coprime to $\Delta$.
\item[$ii)$] There exists a  primitive form $Q_1$ with $[Q] = [Q_1]^2$.
\item[$iii)$] There exist rational numbers $x, y$ with denominator
      coprime to $f$ such that $Q(x,y) = 1$.
\end{enumerate}
Moreover, if $Q$ represents $m^2$ primitively, then $Q_1$ can be 
chosen in such a way that it represents $m$ primitively. 
\end{thm}

\begin{proof}
Observe that Gauss's equation (\ref{EGcom}) implies that if
$Q_1$ represents $m$ and $Q_2$ represents $n$, then 
$Q_3$ represents the product $mn$. Together with the fact
that the primitive form $Q_1$ represents integers coprime to 
$\Delta$ this shows that ii) implies i).

Let us next show that i) and i') are equivalent. It is cleary
sufficient to show that i) implies i'). Assume therefore that 
$Q(x,y) = A^2$ for coprime integers $x, y$, and that there is 
a prime $p \mid \gcd(A,\Delta)$. We claim that $p \mid f$.
We know that $Q$ is equivalent to some form $(A^2,B,C)$, so we 
may assume that $Q = (A^2,B,C)$.

If $p$ is odd, then $p \mid \Delta = B^2 - 4A^2C$, hence $p \mid B$,
$p^2 \mid \Delta$, and finally $p^2 \mid f$ since fundamental
discriminants are not divisible by squares of odd primes.

If $p = 2$, then $B = 2b$ and $A = 2a$, and $(a^2,b,C)$ is a 
form with discriminant $\Delta/4$, showing that $2 \mid f$.
Thus i) and i') are equivalent.

For showing that i') $\Lra$ ii), assume that $Q$ represents
$m^2$ primitively (cancelling squares shows that $Q$ primitively
represents a square), and write $Q = (m^2,B,C)$; Dirichlet 
composition then shows that $[Q] = [Q_1]^2$ for $Q_1 = (m,B,mC)$;
note that if $\Delta < 0$, the form $Q_1$ is positive definite only
for $m > 0$. Since $\gcd(m,\Delta) = \gcd(m,B^2) = 1$, the form 
$Q_1$ is primitive. 

Finally, i) and iii) are trivially equivalent.
\end{proof}

We will also need

\begin{cor}\label{Cgen}
Let $Q_1$ and $Q_2$ be forms with discriminant $\Delta = \Delta_0 f^2$. 
If $Q_j(r_j,s_j) = ax_j^2$ for integers $r_j, s_j, x_j$ ($j = 1, 2$)
with $\gcd(x_1x_2,f) = 1$, then $Q_1$ and $Q_2$ belong to the same genus.
\end{cor}

\begin{proof}
Any form in the class $[Q_1][Q_2]$ represents a square 
coprime to $f$, hence $[Q_1][Q_2] = [Q]^2$. This implies 
that $[Q_1]$ and $[Q_2]$ are in the same genus.
\end{proof}

\subsection*{Nonfundamental Discriminants}

For negative discriminants, Gauss proved a relation between the 
class numbers $h(\Delta)$ and $h(\Delta f^2)$. For general
discriminants, a similar formula was derived by Dirichlet from 
his class number formula, and Lipschitz later gave an arithmetic
proof of the general result. Since we only consider positive
definite forms, we are content with stating a special case of 
Gauss's result:

\begin{thm}
Let $p$ be a prime, and $\Delta < -4$ a discriminant. Then
\begin{equation}\label{ELip}
  h(\Delta p^2) = \Big(p - \Big(\frac{\Delta}p\Big)\Big) \cdot h(\Delta). 
\end{equation}
\end{thm}

The basic tool needed for proving this formula is showing
that every form with discriminant $\Delta p^2$ is equivalent
to a form $(A,Bp,Cp^2)$, which is ``derived'' from the form 
$(A,B,C)$ with discriminant $\Delta$.

Class groups of primitive forms with nonfundamental discriminants
occur naturally in the theory of binary quadratic forms, and
correspond to certain ray class groups (called ring class groups)
in the theory of ideals.

The simplest way of proving (\ref{ELip}) is by using the elementary 
fact that every primitive form with discriminant $\Delta = f^2\Delta_0$
is equivalent to a form $Q = (A,Bf,Cf^2)$. The form $\ov{Q} = (A,B,C)$
is a primitive form with discriminant $\Delta_0$ from which $Q$ is 
derived.

\section{Kaplansky's ``Conjecture''}

Theorem \ref{T2} is related to a question of Kaplansky discussed by
Mollin \cite{MolK1,MolK2} and Walsh \cite{WalK1,WalK2}: Kaplansky
claimed that if a prime $p$ can be written in the form $p = a^2 + 4b^2$,
then the equation $x^2 - py^2 = a$ is solvable. The assumption
$p = a^2 + 4b^2$ implies $p \equiv 1 \bmod 4$, as well as the
solvability of the equation $T^2 - pU^2 = -1$. Since $a^2 = p - 4b^2$,
the form $(1,0,-p)$ with discriminant $\Delta = 4p$ represents $a^2$. 
Since $\gcd(a,4p) = 1$, there is a form $Q$ with discriminant $\Delta$ 
and $[Q^2] = [Q_0]$ which represents $a$. Since the class number 
$h(4p)$ is odd (from (\ref{ELip}) we find that $h(4p) = h(p)$ if
$p \equiv 1 \bmod 8$, and $h(4p) = 3h(p)$ if $p \equiv 5 \bmod 8$; it
is a well known result due to Gauss that the class number of forms with
prime discriminant ist odd), we have $Q \sim Q_0 = (1,0,-p)$, and the 
claim follows.

The obvious generalization of Kaplansky's result is

\begin{prop}
Let $m$ be an integer and $p$ a prime coprime to $2m$. If 
$p = r^2 + ms^2$, then there is a form $Q$ with the following
properties:
\begin{enumerate} 
\item $\disc\ Q = 4pm$;
\item $Q^2 \sim (p,0,-m)$;
\item $Q$ represents $r$.
\end{enumerate}
\end{prop}

\begin{proof}
From $p - ms^2 = r^2$ and $\gcd(r,2pm) = 1$ we deduce that the 
form $Q_1 = (p,0,-m)$ is equivalent to the square of a form 
representing $r$.
\end{proof}

As an example, let $m = 2$ and consider primes $p = r^2 + 2s^2$. 
Then $p-2s^2 = r^2$, hence the form $Q = (p,0,-2)$ with discriminant
$\Delta = 8p$ represents the square number $r^2$. Thus $Q \sim Q_1^2$
for some form $Q_1$ representing $r$. 

Now assume that $p \equiv 1 \bmod 8$. By genus theory, the 
class of the form $Q_1$ will be a square if and only if 
$(\frac2r) = (\frac rp) = 1$. Since $(\frac2r) = (\frac pr)
 = (\frac rp)$ by the quadratic reciprocity law, $[Q_1]$ is
a square if and only if $(\frac rp) = 1$. The congruence
$t^2 \equiv -2s^2 \bmod p$ implies 
$\sym{r^2}{p} = (\frac rp) = \sym{2}{p} (\frac sp)$;
writing $s = 2^ju$ for some odd integer $u$ we find
$(\frac sp) = (\frac 2p)^j (\frac up) = (\frac pu) = 1$.
We have proved (see Kaplan \cite{Kap} for proofs of this and a 
lot of other similar results):

\begin{prop}\label{PKap2}
Let $p = r^2 + 2s^2 \equiv 1 \bmod 8$ be a prime. Then the 
class of the form $Q = (p,0,-2)$ in $\Cl(8p)$ is a fourth power 
if and only if $\sym{2}{p} = +1$.
\end{prop}

Note that this does not necessarily imply that the class number 
$h(8p)$ is divisible by $4$ since $Q$ might be equivalent
to the principal form. In fact, this always happens if 
$r = 1$, since then $Q$ represents $1$. In this case, we get
a unit in $\Z[\sqrt{2p}\,]$ for free because $p-2s^2 = 1$ 
implies that $(\sqrt{p}+s\sqrt{2}\,)^2 = p+2s^2 + 2s\sqrt{2p}$
is a nontrivial unit. Observe that the field is of Richaud-Degert
type since $2p = (2s)^2 + 2$.


\section{P\'epin's Theorems}

The simplest among the about 100 theorems stated by P\'epin
\cite{Pep74, Pep79, Pep80, Pep82} are the following:

\begin{prop}
In the table below, $Q$ is a positive definite form with discriminant 
$\Delta = -4m$. For any prime $p \nmid \Delta$ represented by $Q$
(the table below gives two small values of such $p$), the equations 
$pX^4 - mY^4 = z^2$ only have the trivial solution.
$$ \begin{array}{c|cccccc}
    \rsp   Q & (2,0,7) & (2,2,9)  & (4,0,5)  
             & (4,4,9) & (4,0,9) & (3,0,13) \\ \hline 
    \rsp   m &   14    & 17  &  20 &  32 &  36 & 39 \\ \hline
    \rsp   p & 71, 79  & 13, 89 & 41, 149 & 17,89  & 13,73 & 61, 79  
   \end{array} $$ 
\end{prop}

By looking at these results from the theory of binary quadratic
forms one is quickly led to observe that the equivalence classes
of the forms $Q$ in P\'epin's examples are squares but not fourth
powers. Such forms occur only for discriminants $\Delta = -4m$
for which $\Cl(\Delta)$ has a cyclic subgroup of order $4$. The
table below lists all positive $m \le 238$ with the property that 
$\Cl(-4m)$ has a cyclic subgroup of order $4$, forms $Q$ whose 
classes are squares but not fourth powers, the structure of the 
class group, a comment indicating the proof of the result, and a 
reference to the paper of P\'epin's in which it appears.

\begin{table}[ht!]
   \begin{tabular}{r|l|l|l|c}
      m  &   Q      & $\Cl(-4m)$ & \text{comment} & \text{ref} \\  \hline
     14  & (2,0,7)  & [4]  & \ref{TPepG}.\ref{T51} & \cite{Pep80} \\
     17  & (2,2,9)  & [4]  & \ref{TPepG}.\ref{T51} & \cite{Pep80} \\
     20  & (4,0,5)  & [4]  & \ref{TPepG}.\ref{T54} & \cite{Pep80} \\
     32  & (4,4,9)  & [4]  & \ref{TPepG}.\ref{T55} & \cite{Pep80} \\
     34  & (2,0,17) & [4]  & \ref{TPepG}.\ref{T51} & \cite{Pep80} \\
     36  & (4,0,9)  & [4]  & \ref{TPepG}.\ref{T56} & \cite{Pep74} \\
     39  & (3,0,13) & [4]  & \ref{TPepG}.\ref{T53} & \cite{Pep80} \\
     41  & (5,4,9)  & [8]  & \ref{TPepG}.\ref{T51} & \cite{Pep79} \\
     46  & (2,0,23) & [4]  & \ref{TPepG}.\ref{T51} & \cite{Pep80} \\
     49  & (2,2,25) & [4]  & \ref{TPepG}.\ref{T52}, f=7 & \cite{Pep80} \\
     52  & (4,0,13) & [4]  & \ref{TPepG}.\ref{T54} & \cite{Pep80} \\
     55  & (5,0,11) & [4]  & \ref{TPepG}.\ref{T53} & \cite{Pep80} \\
     56  & (4,4,15) & [4, 2]  & $56 = 4 \cdot 14$ & \cite{Pep80} \\
     62  & (9,2,7)  & [8]  & \ref{TPepG}.\ref{T51} & \cite{Pep79} \\
     63  & (7,0,9)   & [4] & \ref{TPepG}.\ref{T52}, f=3 & \cite{Pep80} \\
     64  & (4,4,17)  & [4] & \ref{TPepG}.\ref{T55} & \cite{Pep80} \\
     65  & (10,10,9) & [4, 2]  & \ref{TPepG}.\ref{T51} & \cite{Pep74} \\
     66  & (3,0,22)  & [4, 2]  & \ref{TPepG}.\ref{T51} & \cite{Pep80} \\
     68  & (8,2,9)\hspace{-1.05cm}--------- (8,4,9)   
                     & [8]     & $68 = 4 \cdot 17$ & \cite{Pep79} \\
     69  & (6,6,13)  & [4, 2]  & \ref{TPepG}.\ref{T51} & \cite{Pep74} \\
     73  & (2,2,37)  & [4]     & \ref{TPepG}.\ref{T51} & \cite{Pep80} \\
     77  & (14,14,9) & [4, 2]  & \ref{TPepG}.\ref{T51} & \cite{Pep80} \\
     80  & (9,2,9)   & [4, 2]  & $80 = 4 \cdot 20$ & \cite{Pep80} \\
     82  & (2,0,41)  & [4]     & \ref{TPepG}.\ref{T51} & \cite{Pep80} \\
     84  & (4,0,21)  & [4, 2]  & \ref{TPepG}.\ref{T54} & \cite{Pep80} \\
     89  & (9, 2, 10), (2,2,45) & [12]   & \ref{TPepG}.\ref{T51} & -- \\
     90  & (9,0,10)  & [4, 2]  & \ref{TPepG}.\ref{T52}, f=3 & \cite{Pep74} \\
     94  & (7,4,14)  & [8] & \ref{TPepG}.\ref{T51} & -- \\ 
     95  & (9,4,11)  & [8] & \ref{TPepG}.\ref{T53} & -- \\
     96  & (4,4,25)  & [4, 2]  & \ref{TPepG}.\ref{T55} & \cite{Pep80} \\
     97  & (2,2,49)  & [4] & \ref{TPepG}.\ref{T51} & \cite{Pep80} \\
     98  & (9,2,11)  & [8] & \ref{TPepG}.\ref{T52}, f=7 & -- \\
    100  & (4,0,25)  & [4] & incorrect & \cite{Pep80} \\
    111  & (7,2,16) & [8]  & \ref{TPepG}.\ref{T53} & \cite{Pep79} \\
    113  & (9,4,13) & [8]  & \ref{TPepG}.\ref{T51} & \cite{Pep74} \\
    114  & (2,0,57)\hspace{-1.18cm}----------- (6,0,19) 
         & [4, 2] & \ref{TPepG}.\ref{T51} & \cite{Pep82} \\
    116  & (9,2,13), (4,0,29) & [12] & \ref{TPepG}.\ref{T54} & -- \\ 
    117  & (9,0,13) & [4, 2] & \ref{TPepG}.\ref{T52},f=3 & \cite{Pep74} \\
    126  & (7,0,18) & [4, 2] & $126 = 9 \cdot 14$ & \cite{Pep80} \\
    128  & (9,8,16) & [8]    & $128 = 4 \cdot 32$ & \cite{Pep74}  \\
    132  & (4,0,33) & [4, 2] & \ref{TPepG}.\ref{T54} & \cite{Pep80} \\
    136  & (8,0,17) & [4, 2] & $136 = 4 \cdot 34$ & \cite{Pep80} \\ 
    137  & (9,8,17) & [8]  & \ref{TPepG}.\ref{T51} & \cite{Pep74} 
    \end{tabular} \bigskip
\caption{Unsolvable Equations $px^4 - my^4 = z^2$.}\label{Tmp}
\end{table}

\begin{table}[ht!]
   \begin{tabular}{r|l|l|l|c}
      m  &   Q      & $\Cl(-4m)$ & \text{comment} & \text{ref} \\  \hline
    138  & (3,0,46) & [4, 2] & \ref{TPepG}.\ref{T51} & \cite{Pep80} \\
    141  & (6,6,25) & [4, 2] & \ref{TPepG}.\ref{T51} & \cite{Pep74} \\
    142  & (2,0,71) & [4]    & \ref{TPepG}.\ref{T51} & \cite{Pep80} \\
    144  & (9, 0, 16) & [4, 2] & $144 = 4 \cdot 36$ & -- \\ 
    145  & (5,0,29) & [4, 2] & \ref{TPepG}.\ref{T51}  & \cite{Pep80} \\
    146  & (3, 2, 49) & [16] & \ref{TPepG}.\ref{T51}  & -- \\ 
    148  & (4,0,37) & [4]    & \ref{TPepG}.\ref{T54}  & \cite{Pep80} \\
    150  & (6,0,25) & [4, 2]  & incorrect & \cite{Pep82} \\
    153  & (13,4,13)\hspace{-1.35cm}------------ (13,8,13)   
         & [4, 2]  & $153 = 9 \cdot 17$ & \cite{Pep82} \\
    154  & (11,0,14) & [4, 2]  & \ref{TPepG}.\ref{T51}  & \cite{Pep82} \\
    155  & (9, 8, 19), (5,0,31)  & [12]& \ref{TPepG}.\ref{T54}   & -- \\ 
    156  & (12,0,13) & [4, 2]  & $156 = 4 \cdot 39$ & \cite{Pep82} \\
    158  & (9,4,18) & [8]  & \ref{TPepG}.\ref{T51}  & \cite{Pep79} \\
    160  & (4,4,41) & [4, 2]  & \ref{TPepG}.\ref{T55}  & \cite{Pep82} \\
    161  & (9, 2, 18) & [8, 2]  & \ref{TPepG}.\ref{T51}   & -- \\ 
    164  & (9, 8, 20) & [16]  & $164 = 4 \cdot 41$ & -- \\ 
    171  & (7, 4, 25), (9,0,19) & [12]  & \ref{TPepG}.\ref{T52}, f=3  & -- \\ 
    178  & (11,6,17) & [8]  & \ref{TPepG}.\ref{T51}  & \cite{Pep79} \\
    180  & (9,0,20)\hspace{-1.18cm}----------- (4,0,45) 
         & [4, 2]  & $180 = 9 \cdot 20$ & \cite{Pep82} \\
    183  & (13,10,16) & [8]  & \ref{TPepG}.\ref{T53} & \cite{Pep79} \\
    184  & (8,8,25) & [4, 2]  & $184 = 4 \cdot 46$ & \cite{Pep82} \\
    185  & (9,4,21) & [8, 2]  & \ref{TPepG}.\ref{T51}  & -- \\ 
    192  & (4,2,29)\hspace{-1.18cm}----------- (4,4,49) 
            & [4, 2]  & \ref{TPepG}.\ref{T55}  & \cite{Pep82} \\
    193  & (2,2,97) & [4]  & \ref{TPepG}.\ref{T51}   & \cite{Pep80} \\
    194  & (11, 4, 18), (6,4,33), (2,0,97) 
                      & [20]  & \ref{TPepG}.\ref{T51} & -- \\ 
    196  & (8,4,25) & [8]  & $196 = 4 \cdot 49$ & \cite{Pep79} \\
    198  & (9,0,22) & [4, 2] & \ref{TPepG}.\ref{T52}, f=3
                                    & \cite{Pep74,Pep82} \\
    203  & (9,4,23), (7,0,29) &  [12]  & \ref{TPepG}.\ref{T53}  & -- \\ 
    205  & (5,0,41) &  [4, 2]  & \ref{TPepG}.\ref{T51}  & \cite{Pep82} \\
    206  & (9, 2, 23) (14,4,15), 2,0,103) 
                       & [20] & \ref{TPepG}.\ref{T51} & -- \\ 
    208  & (16,8,17)\hspace{-1.35cm}------------ (16,16,17)  
         &  [4, 2]  & $208 = 4 \cdot 52$ & \cite{Pep82} \\
    212  & (9, 4,24), (4,0,53) &  [12]  & \ref{TPepG}.\ref{T54}   & -- \\ 
    213  & (6,6,37) &  [4, 2]  & \ref{TPepG}.\ref{T51} & \cite{Pep74} \\
    217  & (2, 2, 109)  & [4, 2]  & \ref{TPepG}.\ref{T51} & \cite{Pep74}  \\
    219  & (12, 6, 19), (3,0,73) & [12] & \ref{TPepG}.\ref{T53}  & -- \\
    220  & (5, 0, 44) & [4, 2] & $220 = 4 \cdot 55$ & \cite{Pep82} \\
    221  & (9,4,25)   & [8, 2] & \ref{TPepG}.\ref{T51} & -- \\
    224  & (9,2,25)   & [8, 2] & \ref{TPepG}.\ref{T55} & -- \\
    225  & (9,0,25)   & [4, 2] & incorrect & \cite{Pep82} \\
    226  & (11,8,22)  & [8]    & \ref{TPepG}.\ref{T51} &  \\
    228  & (4,0,57)   & [4, 2] & \ref{TPepG}.\ref{T54} & -- \\
    233  & (9,2,26), (2,2,117) &  [12]  & \ref{TPepG}.\ref{T51} & -- \\
    238  & (2, 0,119) & [4, 2] & \ref{TPepG}.\ref{T51} &  \cite{Pep82}
    \end{tabular}
\medskip
   \caption{Unsolvable Equations $px^4 - my^4 = z^2$.}\label{Tmp2}
   \end{table}

\begin{table}[ht!]
   \begin{tabular}{r|l|l|l|c}
      m  &   Q      & $\Cl(-4m)$ & \text{comment} & \text{ref} \\  \hline
    265  & (10,10,29)  & [4,2]   & \ref{TPepG}.\ref{T51} & \cite{Pep74} \\
    301  & (14,14,25)  & [4,2]   & \ref{TPepG}.\ref{T51} & \cite{Pep74} \\
    360  & (9,0,40)    & [4,2,2] & $360 = 4 \cdot 90$ & \cite{Pep74} \\
    465  & (10,10,49)  & [4,2,2] & \ref{TPepG}.\ref{T51} & \cite{Pep74} \\
    522  & (9,0,58)    & [4,2]   & \ref{TPepG}.\ref{T51}, f=3 & \cite{Pep74} \\
    553  & (2,2,227)   & [4,2]   & \ref{TPepG}.\ref{T51} & \cite{Pep74} \\
    561  & (34,34,25) $\sim$ (25,16,25) & [4,2,2] & \ref{TPepG}.\ref{T51} 
                                                  & \cite{Pep74} \\
    609  & (42,42,25) $\sim$ (25,8,25)  & [4,2,2] & \ref{TPepG}.\ref{T51} 
                                                  & \cite{Pep74} \\
    645  & (6,6,109)   & [4,2,2] & \ref{TPepG}.\ref{T51} & \cite{Pep74} \\
    697  & (2,2,349)   & [4,2]   & \ref{TPepG}.\ref{T51} & \cite{Pep74} \\
    792  & (9,0,88)    & [4,2,2] & $792=4 \cdot 198$ & \cite{Pep74} \\
   1764  & (4,0,441), (25,12,72), (9,0,196)  
         & [8,4] & $1764 = 42^2$ & \cite{Pep74} \\
   3185  & (9,2,354)   & [16,2,2] & $3185 = 7^2 \cdot 65$ & \cite{Pep74} \\
   4356  & (4,0,1089), (148,96,45), (229,74,25) 
                       & [12,4] & $4256 = 66^2$ & \cite{Pep74} \\
   4950  & (31,28,166), (9,0,550),  
         & [12,2,2] & $4950 = 15^2 \cdot 22$ & \cite{Pep74} \\
   8349  & (25,2,254), (49,36,177), 70,62,133) 
         & [24,2,2] & $8349 = 11^2 \cdot 39$ & \cite{Pep74} \\ \hline
    256  & (16,8,17)  & [8]   & $256 = 4 \cdot 64$ & \cite{Pep79} \\
    289  & (13,12,25) & [8]   & incorrect & \cite{Pep79} \\
    292  & (8,4,37)   & [8]   & $292 = 4 \cdot 73$ &   \cite{Pep79} \\
    295  & (16,6,19)  & [8]   & \ref{TPepG}.\ref{T53} & \cite{Pep79} \\
    313  & (13,10,26) & [8]   & \ref{TPepG}.\ref{T51} & \cite{Pep79} \\ \hline 
    252  & (9,0,28)   & [4,2] & $252 = 4 \cdot 63$ &   \cite{Pep82} \\ 
    282  & (3,0,94)   & [4,2] & \ref{TPepG}.\ref{T51} &   \cite{Pep82} \\ 
    288  & (4,2,73)\hspace{-1.18cm}-----------  (4,4,73) 
                      & [4,2] & $288 = 9 \cdot 32$ &   \cite{Pep82} \\ 
    310  & (10,0,31)  & [4,2] & \ref{TPepG}.\ref{T51} &   \cite{Pep82} \\ 
    322  & (2,0,161)  & [4,2] & \ref{TPepG}.\ref{T51} &   \cite{Pep82} \\ 
    328  & (8,0,41)   & [4,2] & $328 = 4 \cdot 82$ &   \cite{Pep82} \\ 
    333  & (9,0,37)   & [4,2] & $333 = 9 \cdot 37$ &   \cite{Pep82} \\ 
    340  & (4,0,85)   & [4,2] & $340 = 4 \cdot 65$ &   \cite{Pep82} \\ 
    352  & (4,2,89)\hspace{-1.18cm}-----------  (4,4,89)  
                      & [4,2] & \ref{TPepG}.\ref{T55} &   \cite{Pep82} \\ 
    372  & (4,0,93)   & [4,2] & $372 = 4 \cdot 39$ &   \cite{Pep82}  
    \end{tabular}
\medskip
   \caption{More of P\'epin's Examples.}\label{Tmp3}
   \end{table}


P\'epin must have been aware of the connection between his claims 
and the structure of the class group for the following reasons:
\begin{enumerate}
\item The examples he gives in \cite{Pep79} all satisfy
      $\Cl(-4m) = [8]$ (the cyclic group of order $8$), and
      those in \cite{Pep82} satisfy $\Cl(-4m) = [4,2]$.
\item P\'epin omits all values of $m$ for which the class number 
      $h(-4m)$ is divisible by $3$ or $5$, except for three examples
      given in \cite{Pep74}.
\item Most of the misprints in his list concern the middle coefficient
      of the forms $Q$, which is sometimes given as half the correct
      value; a possible explanation is provided by the fact that Gauss 
      used the notation $(A,B,C)$ for the form $Ax^2 + 2Bxy + Cy^2$.
\end{enumerate}

The following result covers all examples in our table:

\begin{thm}\label{TPepG}
Let $m$ be a positive integer, let $Q$ be a quadratic form with 
discriminant $\Delta = -4m$, and assume that $[Q]$ is a square, 
but not a fourth power in $\Cl(\Delta)$. Let $p \nmid \Delta$ 
be a prime represented by $Q$. Then the diophantine equation 
$px^4 - my^2 = z^2$ has nontrivial integral solutions, and if
one of the following conditions is satisfied, $px^4 - my^4 = z^2$ 
has only the trivial solution: 
\begin{enumerate}
\item\label{T51} $\Delta$ is a fundamental discriminant.
\item\label{T52} $\Delta = \Delta_0 f^2$ for some fundamental discriminant
      $\Delta_0$ and an odd squarefree integer $f$ such that 
      $(\frac{\Delta_0}{q}) = -1$ for all $q \mid f$.
\item\label{T53} $\Delta = 4 \Delta_0$, where $\Delta_0 \equiv 1 \bmod 8$ is
      fundamental.  
\item\label{T54} $\Delta = 4 \Delta_0$, where $\Delta_0 =4n$ is
      fundamental and $n \equiv 1 \bmod 4$.  
\item\label{T55} $\Delta = 16 \Delta_0$, where $8 \mid \Delta_0$. 
\item\label{T56} $\Delta = 4f^2 \Delta_0$ for some odd integer $f$, 
      where $\Delta_0 = 4n$ with $n \equiv 1 \bmod 4$, and 
      $(\Delta_0/q) = -1$ for all primes $q \mid f$. 
\end{enumerate}
\end{thm}

\medskip\noindent{\bf Remark 1.}
In \cite{LPep} I have claimed that some proofs can be generalized
to show the unsolvability of equations of the form $px^4 - my^2 = z^2$.
This is not correct: I have overlooked the possibility that 
$\gcd(y,z) \ne 1$ in the proofs given there. In fact, consider the 
equation $71x^4 -14y^2 = z^2$. We find $71 \cdot 3^2 = 5^4 + 14$, 
giving rise to a solution $(x,y,z) = (3,3,75)$ of 
$71 \cdot x^4 = z^2 + 14 y^2$. 
\medskip

\medskip\noindent{\bf Remark 2.}
Studying a few of P\'epin's examples quickly leads to the conjecture
that Theorem \ref{TPepG} holds without any conditions on the discriminant.
This is not true: three of P\'epins ``theorems'' are actually incorrect:
\begin{enumerate}
\item $m = 100$: here $Q = (4,0,25)$ represents the prime $41 = Q(2,1)$, 
      and the equation $41x^4 - 100y^4 = z^2$ has the solution $(5,2,155)$.
\item $m = 150$: here $Q = (6,0,25)$ represents the prime $31 = Q(1,1)$, 
      and the equation $31x^4 - 150y^4 = z^2$ has the solution
      $(5,3,85)$.
\item $m = 289$: here $Q = (13,12,25)$ represents the prime 
      $p = Q(2,-1) = 53$, and the equation $px^4 - 289y^4 = z^2$ has the 
      solution  $(x,y,z) = (17,11,442)$.
\end{enumerate}
Note that, in these examples, the solutions do not satisfy
the condition $\gcd(x,f) = 1$ of Prop. \ref{Pmain} below.

This shows that we have to be careful when trying to generalize 
Thm. \ref{TPepG} to arbitrary discriminants, and that some sort 
of condition (like those in (1) -- (6)) is necessary.

\medskip\noindent{\bf Remark 3.}
The obvious generalization of P\'epin's theorems to nonprime values 
of $p$ does not hold: the form $Q = (2,0,7)$ represents $15 = Q(2,1)$, 
but the diophantine equation $15x^4 - 14y^4 = z^2$ has the nontrivial 
(but obvious) solution $(1,1,1)$.

As in the proof below, we can deduce that $15^2$ is represented by 
the form $Q$; it does not follow, however, that the square roots 
$(3,\pm 2, 5)$ of $(2,0,7)$ represent $15$: in fact, we have 
$15^2 = Q(9,3)$, so the representation is imprimitive, and $Q$ is 
not equivalent to a form with first coefficient $15^2$.
\medskip

\medskip\noindent{\bf Remark 4.}
Some of the examples given by P\'epin are special cases of others.
Consider e.g. the case $m = 80 = 4 \cdot 20$; the derived form of 
$\ov{Q} = (9,-8,4) \sim (4,0,5)$ with discriminant $-4 \cdot 20$ is the
form $Q = (9,-16,16) \sim (9,2,9)$ with discriminant $\Delta = -4m$; 
its class is easily shown to be a square but not a fourth power, but
this is not needed here since every prime represented by $Q$ is also 
represented by $\ov{Q}$, which means that the result corresponding to 
$m = 80$ is a special case of the result for $m = 20$. 

The same thing happens for $m = 68$, $126$, $128$, \ldots; the 
corresponding entries in the tables above are indicated by the 
comment $m = 4 \cdot n$.

The case $m = -56 = -14 \cdot 4$ is an exception: the derived form 
of $(7,0,2) \sim (2,0,7)$ is $Q = (7,0,8)$, whose class is a square but
not a fourth power in $\Cl^+(\Delta)$. Moreover, the primes that $Q$ 
represents are also represented by $(2,0,7)$. 

In addition we have the form $(3,2,19) \sim (3,-4,20)$, and the latter 
form is derived from $(3,-2,5)$. This form generates $\Cl(-4 \cdot 14)$, 
and the square of $(3,-4,20)$ is equivalent to $(8,8,9)$, which underives 
to $(2,4,9) \sim (2,0,7)$. The composition of $(8,8,9)$ and $(7,0,8)$ 
produces a form equivalent to $(4,4,15)$, which is not a square but 
represents $4$.

Observe that the primes $p$ represented by $(4,4,15)$ are congruent
to $3 \bmod 4$, hence the equations $px^2 - 56y^2 = z^2$ only have
solutions with $2 \mid x$ (thus $23x^2 - 56y^2 = z^2$, for example,
has the solution $(x,y,z) = (2,1,6)$). This implies that the 
corresponding quartic $23x^4 - 56y^4 = z^2$ does not have a $2$-adic
solution in a trivial way: $2 \mid x$ implies $4 \mid z$ and $2 \mid y$,
so a simple descent shows that this equation does not have a nontrivial
solution.

\medskip\noindent{\bf Remark 5.}
P\'epin's calculations for $m = 114$ are incorrect: here, the square 
class in $\Cl^+(-4m)$ is generated by $(6,0,19)$, whereas P\'epin uses 
$(2,0,57)$. Perhaps P\'epin went through the forms $(2,0,n)$
with $n \equiv \pm 1 \bmod 8$; these forms are contained in 
square classes if $n$ is prime, or if $n$ has only prime factors
$\equiv \pm 1 \bmod 8$. If the class number $h(-8n)$ is not
divisible by $8$, the class is not a fourth power.

A direct test whether forms $(2,0,p)$ for $p \equiv \pm 1 \bmod 8$ are
fourth powers can be performed as follows: write $p = e^2 - 2f^2$;
then $Q = (2,0,p)$ represents $e^2 = 2f^2 + p$, and $[Q]$ is a fourth 
power if and only if a form with first coefficient $e > 0$ is in the 
principal genus. If $p \equiv 7 \bmod 8$, this happens if and only if
$(\frac{2}e) = +1$, and if $p \equiv 1 \bmod 8$, if and only if 
$(\frac{-2}e) = +1$. A simple exercise using congruences modulo $16$ 
and quadratic reciprocity shows that $(\frac{2}e) = (-1)^{(p+1)/8}$
in the first, and $(\frac{-2}e) = (\frac 2p)^{\phantom{p}}_4$ in the
second case.

$$ \begin{array}{r|rrc}
     p & e & f &  (-1)^{(p+1)/8} \\ \hline
     7 & 3 & 1 &    -1 \\
    23 & 5 & 1 &    -1 \\
    31 & 7 & 3 &    +1 \\
    47 & 7 & 1 &    +1 \end{array} \qquad \qquad  
\begin{array}{r|rrc}
     p &  e & f & (2/p)_4) \\ \hline
    17 &  5 & 2 &    -1 \\
    41 &  7 & 2 &    -1 \\
    73 &  9 & 2 &    +1 \\
    89 & 11 & 4 &    +1  \end{array} $$ 
\medskip       

Thus from Theorem \ref{TPepG} we get the following

\begin{cor}
Let $m = 2q$, where $q \equiv 7 \bmod 16$ is a prime.
For any prime $p$ represented by $(2,0,q)$, the 
diophantine equation $px^4 - my^4 = z^2$ has only the
trivial solution. The same result holds for primes 
$q \equiv 1 \bmod 8$ with $(2/p)_4 = -1$.
\end{cor}

It is an easy exercise to deduce countless other families
of similar results from Theorem \ref{TPepG}.

\subsection*{The Proof of Theorem \ref{TPepG}}

The part of Thm. \ref{TPepG} which is easiest to prove generalizes 
Remark 1:

\begin{lem}
If $p$ is represented by a form $Q$ in the principal genus of 
$\Cl^+(-4m)$, then the equation $px^4 - my^2 = z^2$ has a nonzero 
integral solution.
\end{lem}

\begin{proof}
The solvability of $px^2 = z^2 + my^2$ follows from genus theory:
the prime $p$ is represented by a form in the principal genus,
hence any form in the principal genus (in particular the principal
form $(1,0,-m)$) represents $p$ rationally. Multiplying through by 
$x^2$ we find $px^4 - m(xy)^2 = (xz)^2$.
\end{proof}

The main idea behind the proof of Thm. \ref{TPepG} is the content 
of the following

\begin{prop}\label{Pmain}
Assume that $Q$ is a form with discriminant $\Delta = -4m$ whose 
class is a square but not a fourth power in $\Cl^+(\Delta)$. Let 
$p \nmid \Delta$ be a prime represented by $Q$, and write 
$\Delta = \Delta_0 f^2$, where $\Delta_0$ is a fundamental 
discriminant. Then the diophantine equation $px^4 - my^4 = z^2$ 
does not have an integral solution $(x,y,z)$ with $\gcd(x,f) = 1$.
\end{prop}

\begin{proof}
The form $Q_0 = (1,0,m)$ represents $px^4 = Q_0(z,y^2)$, and $Q$ 
represents $p$. By Cor. \ref{C1}, the square $p^2x^4$ is represented 
primitively by any form in the class $[Q][Q_0]^e = [Q]$ (for some 
$e = \pm 1$), hence by $Q$ itself. Thus there is a form $Q_1$ with 
$Q_1^2 \sim Q$ representing $px^2$, which by Cor. \ref{Cgen} is in 
the same genus as $Q$. But then $Q_1$ and $Q$ differ by a square
class, which implies that $[Q_1]$ is a square and that $[Q]$
is a fourth power: contradiction.
\end{proof}

Now we can give the 
\begin{proof}[Proof of Thm. \ref{TPepG}]
We have to show that if $px^4 - my^4 = z^2$ has a nontrivial integral
solution, then there is a solution satisfying $\gcd(x,f) = 1$. Applying
Prop. \ref{Pmain} then gives the desired result.

\subsection*{Case (1)} If $\Delta$ is fundamental, then $f = 1$,
and the conditiong $\gcd(x,f) = 1$ is trivially satisfied.
\subsection*{Case (2)} 
Write $\Delta = 4m$ and $m = q^2n$; then $\Delta_0 = 4n$.
Assume that $px^4 - my^4 = z^2$. If $q \mid x$ and $q \nmid y$, 
then $x = qX$ and $z = qZ$ gives $pq^2X^4 - ny^4 = Z^2$. 
Reduction modulo $q$ implies $(\frac{-n}q) = +1$ contradicting 
our assumptions since $q$ is odd.
\subsection*{Case (3)} If $\Delta = 4 \Delta_0$ with 
$\Delta_0 \equiv 1 \bmod 8$, then we cannot exclude the possibility
that $x$ might be even. Thus in order to guarantee the applicability 
of Prop. \ref{Pmain} we have to get rid of the factor $4$ in 
$\Delta = -4m$. This is achieved in the following way.

By (\ref{ELip}), we have $h(-4m) = (2 - (\frac{-m}2)) h(-m) = h(m)$
if $-m \equiv 1 \bmod 8$. Since there is a natural surjective projection
$\Cl(-4m) \lra \Cl(-m)$, this implies that $\Cl(-4m) \simeq \Cl(-m)$. 
But then the form $Q$ projects to a form $\ov{Q}$ with discriminant $-m$
whose class is a square but not a fourth power. Moreover, every integer
represented by $Q$ is also represented by $\ov{Q}$. An application of 
Prop. \ref{Pmain} with $\Cl^+(\Delta)$ replaced by $\Cl^+(\Delta_0)$
now gives the desired result.

\subsection*{Case (4)} Here $m = 4n$ for $n \equiv 1 \bmod 4$. 
If $px^4 - 4ny^4 = z^2$ and $x = 2X$, then $z = 2z_1$ and 
$4pX^4 - ny^4 = z_1^2$. From 
$0 \equiv z_1^2 + ny^4 \equiv z_1^2 + y^4 \bmod 4$ we find that 
$z_1 = 2Z$ and $y = 2Y$ must be even, hence $pX^4 - 4nY^4 = Z^2$.
Repeating this if necessary we find a solution with $x$ odd.

\subsection*{Case (5)} Here $m = 16n$ with $n \equiv 2 \bmod 4$.
If $px^4 - 16ny^4 = z^2$, then $x = 2X$ and $z = 4Z$, hence
$pX^4 - nY^4 = Z^2$. 

Since $p$ is represented by a form in the principal genus,
we must have $pa^2 = b^2 + mc^2$ for some odd integer $a$.
This implies $p \equiv 1+mc^2 \equiv 1 \bmod 8$.

From $1 - 2Y^4 \equiv 1 \bmod 4$ we deduce that $Y = 2y_1$,
hence $pX^4 - 16ny_1^4 = Z^2$. This proves our claim.

\subsection*{Case (6)} Assume that $m = 4f^2n$ with $n \equiv 1 \bmod 4$ 
and $(-n/q) = -1$ for all primes $q \mid f$, and that $px^4 - my^4 = z^2$. 
If $x = 2X$ and $z = 2z_1$, then $4pX^4 - f^2ny^4 = z_1^2$. If $y$ is odd, 
then $z_1^2 \equiv 3 \bmod 4$: contradiction. Thus $y = 2Y$ and $z = 2Z$, 
hence $pX^4 - 4f^2nY^4 = Z^2$.

If $f = qg$ and $x = qX$, $z = qz_1$, then $pq^2X^4 - 4g^2ny^4 = z_1^2$.
Reduction mod $q$ gives $(-n/q) = +1$ contradicting our assumption,
except when $q \mid y$. But then $y = qY$ and $z_1 = qZ$, and we find 
$pX^4 - mY^4 = Z^2$.

\end{proof}

The abstract argument in case (3) above can be made perfectly explicit:

\begin{lem}
Assume that $m \equiv 3 \bmod 4$, and that $Q = (A,B,C)$ is a   
form  with discriminant $B^2 - 4AC = -4m$. We either have
\begin{enumerate}
\item $4 \mid B$;
\item $B \equiv 2 \bmod 4$ and $4 \mid C$; or
\item $B \equiv 2 \bmod 4$ and $4 \mid A$.
\end{enumerate}
Then every integer represented by $Q$ is also represented by the form 
$$ Q' = \begin{cases}
    (A, A+\frac{B}2,\frac{A+B+C}4)  & \text{ in cases (1) and (2)} \\
    (\frac{A}4,\frac{B}2,C)         & \text{ in case (3)}
        \end{cases} $$          
with discriminant $-m$. 
\end{lem}

\begin{proof}
From $B^2 - 4AC = -4m$ we get $(\frac B2)^2 - AC = -m \equiv 1 \bmod 4$.

If $4 \mid B$, then $AC \equiv 3 \bmod 4$, which implies 
$A + C \equiv 0 \bmod 4$. Thus in this case, $\frac B2$ and 
$\frac{A+B+C}4$ are integers.

If $2 \parallel B$, then $1-AC \equiv 1 \bmod 4$ implies that $4 \mid AC$.
Since the form $Q$ is primitive, we either have $4 \mid A$ or $4 \mid C$.

In cases (1) and (2), substitute $x = X + \frac 12 Y$ and $y = \frac12 Y$
in $Q(x,y) = Ax^2 + Bxy + Cy^2$ and observe that if $x$ and $y$ are 
integers, then so are $X$ and $Y$. We find
$Q(x,y) = A(X + \frac 12 Y)^2 + B(X + \frac 12 Y)(\frac 12 Y) 
          + C(\frac 12 Y)^2 = AX^2 + (A + \frac B2)XY + \frac{A+B+C}4 Y^2$.
The coefficients of $Q'$ in these two cases are clearly integers.
In case (3), set $x = \frac12 X$ and $y = Y$.
\end{proof}

\medskip\noindent
{\bf Remark.} The proof we have given works with forms in 
$\Cl(\Delta)^2$; for distinguishing simple squares from 
fourth powers, one can introduce characters on  
$\Cl(\Delta)^2/\Cl(\Delta)^4$. This was first done in 
a special case by Dirichlet, who showed how to express 
these characters via quartic residue symbols; much later, his
results were generalized within the theory of spinor genera by 
Estes \& Pall \cite{EP}. This explains why most proofs of the 
nonsolvability of equations of the form $ax^4 + by^4 = z^2$
(see in particular \cite{LFC}) use quartic residue symbols.

\section{Examples of Lind-Reichardt type}

The most famous counterexample to the Hasse principle is due to 
Lind \cite{Lind} and Reichardt \cite{Reich}; Reichardt showed that the 
equation $17x^4 - 2y^2 = z^4$ has solutions in every localization
of $\Q$, but no rational solutions except the trivial $(0,0,0)$,
and Lind constructed many families of similar examples. Below,
we will show that our construction also gives some of their
examples; we will only discuss the simplest case of fundamental
discriminants and are content with the remark that there are
similar results in which $-4AC$ is not assumed to be fundamental.
 
\begin{thm}
Let $A$ and $C$ be coprime positive integers such that 
$\Delta = -4AC$ is a fundamental discriminant.
Then $Q = (A,0,C)$ is a primitive form with discriminant $\Delta$. 
If the equivalence class of $Q$ is a square but not a fourth power in 
$\Cl^+(\Delta)$, then the diophantine equation $Ax^4 - z^4 = Cy^2$ 
only has the trivial solution $(0,0,0)$.
\end{thm}

\begin{proof}
We start with the observation that the form $(A,0,-C)$ represents $z^4$. 
Since $4AC$ is fundamental, the class $[Q]$ must be a fourth power, 
contradicting our assumptions. 
\end{proof}

The following result is very well known:

\begin{cor}
Let $p \equiv 1 \bmod 8$ be a prime with $\sym{2}{p} = -1$.
Then the diophantine equation $px^4 - 2y^2 = z^4$ does not have
a trivial solution.
\end{cor}

\begin{proof}
The form $Q = (p,0,-2)$ has discriminant $8p$, and Prop. \ref{PKap2}
implies that its class is a square but not a fourth power in $\Cl^+(8p)$.
\end{proof}

\section{Hasse's Local-Global Principle}

Some diophantine equations $ax^4 + by^4 = z^2$ can be proved to 
have only the trivial solution by congruences, that is, by studying 
solvability in the localizations $\Q_p$ of the rationals. A special 
case of Hasse's Local-Global Principle asserts that quadratic
equations $ax^2 + by^2 = z^2$ have nontrivial solutions
in integers (or, equivalently, in rational numbers) if and
only if it has solutions in every completion $\Q_p$.

It is quite easy to see (cf. \cite{AL}) that $ax^4 + by^4 = z^2$ 
(and therefore also $ax^2 + by^2 = z^2$) has local solutions for 
every prime $p \nmid 2ab$, so it remains to check solvability in 
the reals $\R = \Q_\infty$ and in the finitely many $\Q_p$ with 
$p \mid 2ab$.

Actually we find

\begin{prop}
If $px^2 - my^2 = z^2$ has a rational solution, then 
the quartic $px^4 - my^4 = z^2$ has a solution in 
$\Z_q$ for every prime $q > 2$.
\end{prop}

\begin{proof}
By classical results (see \cite{AL} for a very simple and elementary
proof), $px^4 - my^4 = z^2$ is locally solvable for every prime 
$q \nmid 2pm$. Thus we only have to look at odd primes $q \mid mp$.

From the solvability of $px^2 - my^2 = z^2$ we deduce that 
$(\frac pq) = +1$ for every odd prime $q \mid m$. But then 
$\sqrt{p} \in \Z_q$, and we can solve $px^4 - my^4 = z^2$ 
simply by taking $(x,y,z) = (1,0,\sqrt{p})$.
Moreover, from $(\frac{-m}p) = +1$ we deduce that $\sqrt{-m} \in \Z_p$,
hence $(x,y,z) = (0,1,\sqrt{-m}\,)$ is a $p$-adic solution of
$px^4 - my^4 = z^2$.
\end{proof}

Thus the quartic $px^4 - my^4 = z^2$, where $p, m > 0$, will be 
everywhere locally solvable if and only if it is solvable in the 
$2$-adic integers $\Z_2$. Now we claim

\begin{prop}
If $m \equiv 2 \bmod 4$, P\'epin's equations are solvable in $\Z_2$,
hence have solutions in every completion of $\Q$.
\end{prop}

\begin{proof}
Assume first that $m \equiv 2 \bmod 8$. Then $Q_0 = (1,0,m)$ 
(and therefore any form in the principal genus) represents 
primes $p \equiv 1, 3 \bmod 8$. If $p \equiv 1 \bmod 8$,
then $p \cdot 1^4 - m \cdot 0^4 \equiv 1 \bmod 8$ is a square in
$\Z_2$, and if $p \equiv 3 \bmod 8$, then 
$p\cdot 1^4 - m \cdot 1^4 \equiv 3 - 2 \equiv 1 \bmod 8$
is a square in $\Z_2$.

Assume next that $m \equiv 6 \bmod 8$. 
Then $Q_0 = (1,0,m)$ represents only primes $p \equiv \pm 1 \bmod 8$;
the same holds for all forms in the principal genus, in particular
the form $Q$ represents only primes $p \equiv \pm 1 \bmod 8$.
For primes $p \equiv 1 \bmod 8$, the element $p\cdot 1^4 - m \cdot 0^4$
is a square in $\Z_2$, so $px^4 - my^4 = z^2$ is solvable in $\Z_2$.
If $p \equiv 7 \bmod 8$, then  
$p\cdot 1^4 - m \cdot 1^4 = p-m \equiv 7 - 6 \equiv 1 \bmod 8$
is a square in $\Z_2$.
\end{proof}

A similar but simpler proof yields

\begin{prop}
If $m \equiv 0 \bmod 8$, P\'epin's equations are solvable in $\Z_2$,
hence have solutions in every completion of $\Q$.
\end{prop}

In fact, primes represented by a form in the principal genus are
$\equiv 1 \bmod 8$, and for these, showing the $2$-adic solvability 
is trivial. Similarly we can show

\begin{prop}
If $m \equiv 1 \bmod 8$, then $p \equiv 1 \bmod 4$, and P\'epin's 
equations are solvable in $\Z_2$ (hence in every completion of $\Q$) 
if $p \equiv 1 \bmod 8$.
\end{prop}

\begin{proof}
Since $p$ is represented by some form in the principal genus, 
we must have $pt^2 = r^2 + ms^2$ for integers $r, s, t$, with
$t$ coprime to $4m$. This implies that $t$ is odd, and that 
$r$ or $s$ is even, hence we must have $p \equiv pt^2 \equiv 1 \bmod 4$
as claimed.

We have to study the solvability of $px^4 - my^4 = z^2$ in $\Z_2$.
If $p \equiv 1 \bmod 8$, then $\sqrt{p} \in \Z_2$, and we have
the $2$-adic solution $(x,y,z) = (1,0,\sqrt{p}\,)$.
\end{proof}

Thus although not all of P\'epin's examples are counterexamples
to Hasse's Local-Global Principle, his construction gives infinite
families of equations $px^4 - my^4 = z^2$ which have local solutions
everywhere but only the trivial soltion in integers. As is well known, 
such equations represent elements of order $2$ in the Tate-Shafarevich
group of certain elliptic curves, and it is actually quite easy
to use P\'epin's construction to find Tate-Shafarevich groups with
arbitrarily high $2$-rank (see \cite{LGr} for the history and a direct 
elementary proof of this result).

\section*{Acknowledgement}

I would like to thank the referee for several helpful comments.


\begin{thebibliography}{99}

\bibitem[AL]{AL} W. Aitken, F. Lemmermeyer,
{\em Counterexamples to the Hasse Principle: An Elementary Introduction},
Amer. Math. Monthly (2011), to appear
%

\bibitem[Bh]{Bhar1} M. Bhargava,
{\em Gauss composition and generalizations},
Algorithmic number theory (Sydney, 2002),  1--8, 
Lecture Notes in Comput. Sci., 2369, Springer, Berlin, 2002
%

\bibitem[C1]{Cay45} A. Cayley,
{\em On the theory of linear transformations},
Cambridge Math. J. {\bf 4} (1845), 1--16;
Coll. Math. Papers I (1889), 80--94
%

\bibitem[C2]{CayHD} A. Cayley,
{\em M\'emoire sur les hyperd\'eterminants},
J. Reine angew. Math. {\bf 30} (1846), 1--37
%

\bibitem[Co]{Cox} D. Cox,
{\em Primes of the Form $x^2 + ny^2$},
Wiley 1989
%

\bibitem[De]{Dede} R. Dedekind,
{\em \"Uber trilineare Formen und die Komposition der bin\"aren
      quadratischen Formen}, 
J. Reine Angew. Math. {\bf 129} (1905), 1--34
%

\bibitem[D1]{Dir25} P.G.L. Dirichlet,
{\em M\'emoire sur l'impossibilit\'e de quelques \'equations
     in\-d\'e\-ter\-mi\-n\'ees du cinqui\`eme degr\'e}, 1825;
        Werke I, 1--20
%

\bibitem[D2]{Dir28} P.G.L. Dirichlet,
{\em M\'emoire sur l'impossibilit\'e de quelques \'equations
     in\-d\'e\-ter\-mi\-n\'ees du cinqui\`eme degr\'e},
        J. Reine Angew. Math. {\bf 3} (1828), 354--375;
        Werke I, 21--46
%

\bibitem[EP]{EP} D.R. Estes, G. Pall,
{\em Spinor genera of binary quadratic forms},
J. Number Theory {\bf 5} (1973), 421--432
%

\bibitem[Fl]{Flath} D. Flath,
{\em Introduction to number theory},
Wiley \& Sons, New York, 1989
%

\bibitem[Ga]{Gauss} C.F. Gauss,
{\em Disquisitiones Arithmeticae},
Leipzig 1801; French transl. by Poullet Delisle (1807); 
reprints 1910, 1953; German transl. by H.~Maser (1889);
English transl. by A.A.~Clarke (1965); 
2nd rev. ed. Waterhouse et al. (1986);
Spanish transl. by H.~Barrantes Campos, M.~Josephy and 
\'A.~Ruiz Z\'u\~niga (1995)
%

\bibitem[GKZ]{GKZ} I.M. Gelfand, M.M. Kapranov, A.V. Zelevinsky,
{\em Discriminants, resultants, and multidimensional determinants},
Birkh\"auser Boston, 1994
%

\bibitem[Hil]{HZB-f} D. Hilbert,
{\em Th\'eorie des corps de nombres alg\'ebriques},
French transl. by Levy and Got, 
Ann. Fac. Sci. Toulouse {\bf 1} (1909), 257--328;
{\bf 2} (1910), 225--456; {\bf 3} (1911), 1--62
%

\bibitem[Ka]{Kap} P. Kaplan,
{\em Sur le $2$-groupe des classes d'id\'eaux des corps quadratiques}, 
J. Reine Angew. Math. {\bf 283/284} (1976), 313--363 
%

\bibitem[Kn]{KneC} M. Kneser, 
{\em Composition of binary quadratic forms}, 
J. Number Theory {\bf 15} (1982), no. 3, 406--413
%

\bibitem[Ko]{Koe87} M. Koecher,
{\em On endomorphisms of degree $2$},
Proc. Indian Acad. Sci. {\bf 97} (1987), 179--188
%

\bibitem[L1]{LPep} F. Lemmermeyer,
{\em A note on P\'epin's counter examples to the Hasse principle
	for curves of genus $1$},
Abh. Math. Sem. Hamburg {\bf 69} (1999), 335--345
%

\bibitem[L2]{LGr} F. Lemmermeyer,
{\em On Tate-Shafarevich groups of some elliptic curves},
Proc. Conf. Graz 1998, (2000), 277--291
%

\bibitem[L3]{LFC} F. Lemmermeyer,
{\em Some families of non-congruent numbers},
Acta Arith. {\bf 110} (2003), 15--36
%

\bibitem[Le]{Len} H.W. Lenstra,
{\em On the calculation of regulators and class numbers of quadratic fields},
Number theory days, 1980 (Exeter, 1980),  123--150, Cambridge 1982
%

\bibitem[Li]{Lind} C.-E. Lind,
{\em Untersuchungen \"uber die rationalen Punkte der ebenen 
ku\-bi\-schen Kurven vom Geschlecht Eins},
Diss. Univ. Uppsala 1940
%

\bibitem[M1]{MolK1} R. Mollin,
{\em Proof of some conjectures by Kaplansky},
C.R. Math. Rep. Sci. Canada {\bf 23} (2001), 60--64
%

\bibitem[M2]{MolK2} R. Mollin,
{\em On a generalized Kaplansky conjecture},
Int. J. Contemp. Math. Sciences {\bf 2} (2007), 411--416
%

\bibitem[P1]{Pep74} Th. P\'epin,
{\em Th\'eor\`emes d'analyse ind\'etermin\'ee},
C. R. Acad. Sci. Paris {\bf 78} (1874), 144--148
%

\bibitem[P2]{Pep75} T. P\'epin,
{\em Sur certains nombres complexes compris dans la formule
   $a + b \sqrt{-c}$},
J. Math. Pures Appl. (3) {\bf I} (1875), 317--372
%

\bibitem[P3]{Pep79} Th. P\'epin,
{\em Th\'eor\`emes d'analyse ind\'etermin\'ee},
C. R. Acad. Sci. Paris {\bf 88} (1879), 1255--1257
%

\bibitem[P4]{PepC} Th. P\'epin, 
{\em Composition des formes quadratiques binaires},
Atti Acad. Pont. Nuovi Lincei {\bf 33} (1879/80), 6--73
%

\bibitem[P5]{Pep80} Th. P\'epin,
{\em Nouveaux th\'eor\`emes sur l'\'equation ind\'etermin\'ee
	$ax^4 + by^4 = z^2$},
C. R. Acad. Sci. Paris {\bf 91} (1880), 100--101
%

\bibitem[P6]{Pep82} Th. P\'epin,
{\em Nouveaux th\'eor\`emes sur l'\'equation ind\'etermin\'ee
	$ax^4 + by^4 = z^2$},
C. R. Acad. Sci. Paris {\bf 94} (1882), 122--124
%

\bibitem[Re]{Reich} H. Reichardt,
{\em Einige im Kleinen \"uberall l\"osbare, im Gro\ss{}en unl\"osbare
  diophantische Gleichungen}, 
J. Reine Angew. Math. {\bf 184} (1942), 12--18
%

\bibitem[Ri]{Riss}  J. Riss,
{\em La composition des formes quadratiques binaires
     (d'apr\`es Gauss)},
S\'em. Th\'eor. Nombr. Bordeaux (1978), exp. 18, 16pp
%

\bibitem[Sch]{Schoof} R. Schoof, 
{\em  Quadratic fields and factorization},
Computational methods in number theory, Part II,  235--286, 
Math. Centre Tracts {\bf 155}, Amsterdam 1982
%

\bibitem[S1]{Sha3} D. Shanks,
{\em A matrix underlying the composition of quadratic forms
     and its implications for cubic extensions},
Notices Amer. Math. Soc. {\bf 25} (1978), p. A305
%

\bibitem[S2]{Shac1} D. Shanks,
{\em On Gauss and Composition I},
in {\em Number Theory and Applications} 
(R. Mollin, ed.), 1989, 163--178
%

\bibitem[S3]{Shac2} D. Shanks,
{\em On Gauss and Composition II},
in {\em Number Theory and Applications} 
(R. Mollin, ed.), 1989, 179--204
%

\bibitem[Sp]{Spei} A. Speiser,
{\em \"Uber die Komposition der bin\"aren quadratischen Formen},
Weber-Festschrift (1912), 375--395
%

\bibitem[To]{Tow} J. Towber, 
{\em Composition of oriented binary quadratic form-classes
      over commutative rings},
     Adv. Math. {\bf 36} (1980), 1--107
%

\bibitem[W1]{WalK1} G. Walsh,
{\em On a question of Kaplansky},
Amer. Math. Monthly {\bf 109} (2002), no. 7, 660--661
%

\bibitem[W2]{WalK2} G. Walsh,
{\em On a question of Kaplansky II},
Albanian J. Math. {\bf 2} (2008), 3--6
%

\bibitem[We]{WebC} H. Weber, 
{\em \"Uber die Komposition der quadratischen Formen},
G\"ott. Nachr. (1907), 86--100
%

\end{thebibliography}
\end{document}